\documentclass[12pt,oneside,reqno]{amsart}

\usepackage{hyperref}

\usepackage[headsep=30pt,footskip=30pt,twoside=false,bottom=2.5cm]{geometry}

\usepackage[all]{xy}

\xymatrixcolsep{2cm}

\numberwithin{equation}{section}

\allowdisplaybreaks[1]

\usepackage{etoolbox}

\makeatletter

\patchcmd{\thesubsection}{\arabic}{\arabic}{}{}

\patchcmd{\@seccntformat}{\@secnumfont}{%
  \@secnumfont\expandafter\protect\csname format#1\endcsname}{}{}

\patchcmd{\@startsection}{\@afterindenttrue}{\@afterindentfalse}{}{}

\patchcmd{\subsection}{-.5em}{.3\linespacing}{}{}

\makeatother

\theoremstyle{plain}

\newtheorem{theorem}{Theorem}[section]

\newtheorem{proposition}[theorem]{Proposition}
\newtheorem{definition}[theorem]{Definition}

\newtheorem{lemma}[theorem]{Lemma}

\newtheorem{corollary}[theorem]{Corollary}

\theoremstyle{remark}

\newcommand{\Ker}[1]{\ensuremath{\mathrm{Ker} (#1)}}

\newcommand{\Pic}[1]{\ensuremath{\mathrm{Pic} (#1)}}

\newcommand{\At}[1]{\ensuremath{\mathrm{At} (#1)}}

\newcommand{\Aut}[1]{\ensuremath{\mathrm{Aut} (#1)}}

\newcommand{\ad}[1]{\ensuremath{\mathrm{ad}  (#1)}}

\newcommand{\ENd}[1]{\ensuremath{\mathrm{End}  (#1)}}

\newcommand{\Img}[1]{\ensuremath{\mathrm{Im} (#1)}}

\newcommand{\cat}[1]{\ensuremath{\mathcal{#1}}}

\renewcommand{\dim}[2][]{\ensuremath{\mathrm{dim}_{#1}(#2)}}

\newcommand{\V}{\ensuremath{\mathbb{V}}}

\newcommand{\W}{\ensuremath{\mathbb{W}}}

\newcommand{\M}{\ensuremath{\mathbb{M}}}
\newcommand{\Z}{\ensuremath{\mathbb{Z}}}

\newcommand{\A}{\ensuremath{\mathbb{A}}}

\newcommand{\p}{\ensuremath{\mathbf{P}}}

\newcommand{\C}{\ensuremath{\mathbb{C}}}

\newcommand{\struct}[1]{\ensuremath{\mathcal{O}_{#1}}}

\newcommand{\DifF}[3][]{%
  \ensuremath{\mathcal{D}\mathit{iff}^{#1}(#2, #3)}}

\newcommand{\coh}[3]{\ensuremath{\mathrm{H}^{#1}(#2,#3)}}

\begin{document}

\title[Moduli space of $\lambda$-connections]{On the Moduli space of $\lambda$-connections }

\author{Anoop Singh}
\address{Harish-Chandra Research Institute, HBNI \\ Chhatnag Road \\ Jhusi \\
  Prayagraj 211~019 \\ India}
  \email{anoopsingh@hri.res.in}

\subjclass[2010]{14D20, 14C22, 14E05, 14J50}
  \keywords{$\lambda$-connection, Moduli space, Picard 
  group, algebraic function, automorphism}
\begin{abstract}
Let $X$ be a compact Riemann surface of genus $g \geq 
3$. Let $\cat{M}_{Hod}$ 
denote the moduli space of stable $\lambda$-connections over $X
$ and $\cat{M}'_{Hod} 
\subset \cat{M}_{Hod}$ denote the subvariety whose
underlying vector bundle is stable. Fix a line
bundle $L$ of degree zero. Let $\cat{M}_{Hod}(L)$
denote the moduli space of stable $\lambda$-connections with fixed determinant 
$L$ and $\cat{M}'_{Hod}(L) \subset \cat{M}_{Hod}(L)$
be the subvariety whose underlying  vector bundle
is stable. We show that there is a natural 
compactification of $\cat{M}'_{Hod}$ and $\cat{M}'_{Hod} 
(L)$, and study their Picard groups.
Let $\M_{Hod}(L)$ denote the moduli space
of polystable $\lambda$-connections. We investigate the 
nature of algebraic functions on $\cat{M}_{Hod}(L)$ and 
$\M_{Hod}(L)$. We also study the automorphism group
of $\cat{M}'_{Hod}(L)$.
 \end{abstract}

\maketitle
\section{Introduction and statements of the results}
\subsection{Notations and conventions}
In what follows, $\A^1$ denote the affine line over
the field of complex numbers $\C$, $X$ denote the 
compact connected Riemann surface of genus $g \geq 3$.
Unless otherwise stated genus $g$ of $X$ is considered to be $\geq 3$.  $J(X)$ denote the Jacobian 
of $X$.

\subsection{Motivation and results}
The Picard group of a moduli space is a very important 
invariant while studying the classification
problems for algebro-geometric objects.  The Picard 
group of moduli space of vector bundles have been 
studied  extensively by several algebraic geometers,
for instance \cite{DN}, \cite{R}, \cite{B} to name a few. Also, the Picard group and algebraic functions 
for the moduli space of logarithmic connections have
been studied in \cite{BR} and \cite{S}.
So, the purpose of this article is to study the Picard
group, algebraic function and automorphism group of the 
moduli spaces of $\lambda$-connections.  
The notion of $\lambda$-connection in a vector bundle
was introduced by P. Deligne \cite{D} realising the 
Higgs bundles as a degeneration of vector bundles 
with connections, where $\lambda$ varies over the field 
of complex numbers $\C$.
In the series of papers \cite{Si1}, \cite{Si4} and \cite{Si5}, Simpson 
studied the moduli space of $\lambda$-connections extensively, where he named the moduli space of 
$\lambda$-connections as Hodge moduli space.

 In section \eqref{pre}, we recall
the notion of $\lambda$-connection in a holomorphic vector bundle $E$ over $X$ of rank $n$. We also recall the 
notion of stable and polystable 
$\lambda$-connections.
 Let $ \cat{M}_{Hod} = 
\cat{M}_{Hod}(X,n)$ denote the moduli space of 
 stable $\lambda$-connections on $X$ of rank 
 $n$. Let $
 \cat{M}'_{Hod} \subset \cat{M}_{Hod}$ be the 
 moduli space whose underlying  vector bundle 
 is stable. We show that there is a natural 
 compactification  of the moduli space $\cat{M}'_{Hod}$
 [see Proposition \eqref{thm:1}].

 Let $\cat{U} = \cat{U}(X,n)$ denote the moduli space of 
 stable vector bundle of rank $n$ and degree $0$, and let 
 \begin{equation}
 \label{eq:0.2}
 p : \cat{M}'_{Hod} \to \cat{U}
  \end{equation}
be the natural projection sending 
$(E, \lambda, D) \mapsto E$.
The morphism $p: \cat{M}'_{Hod} \to \cat{U}$ defined in
\eqref{eq:0.2} induces a homomorphism 
\begin{equation}
\label{eq:0.3}
p^*: \Pic {\cat{U}} \to \Pic{\cat{M}'_{Hod}}
\end{equation}
of Picard groups that sends any line bundle $\xi$ over 
$\cat{U}$ to a line bundle $p^* \xi$ over
 $\cat{M}'_{Hod}$.
Next, let $\iota: \cat{M}'_{Hod} \hookrightarrow \cat{M}_{Hod}$ be the inclusion morphism. Then $\iota$ induces a homomorphism
\begin{equation}
\label{eq:0.4}
\iota^* : \Pic{\cat{M}_{Hod}} \to \Pic{\cat{M}'_{Hod}}
\end{equation}
of Picard groups which restricts any line bundle 
$\xi$ over $\cat{M}_{Hod}$ to $\cat{M}'_{Hod}$.
We show that $p^*$ and $\iota^*$ are isomorphisms of 
groups, more concretely [see Theorem \eqref{thm:2}]
\begin{theorem}
\label{thm:0.2} Let $X$ be a compact Riemann surface of 
genus $g \geq 3$. Then,
the two homomorphism 
\begin{equation*}
\label{eq:0.5}
\Pic{\cat{U}} \xrightarrow{p^*} \Pic{\cat{M}'_{Hod}}
\xleftarrow{\iota^*} \Pic{\cat{M}_{Hod}}
\end{equation*}
defined in \eqref{eq:0.3} and \eqref{eq:0.4} are isomorphisms.
\end{theorem} 

Let $\M_{Hod} = \M_{Hod}(X,n)$ (see \cite{HZ}) denote the coarse moduli 
space of isomorphism classes of polystable 
$\lambda$-connections on $X$ of rank $n$. 
Then we show that $\Pic{\M_{Hod}} \cong \Pic{\cat{M}_{Hod}}$ [see Proposition \eqref{prop:0}].

Fix a line bundle $L$ over  
$X$ of degree zero. For each $\lambda \in \C$, fix a $\lambda$-connection $D^{\lambda}_{L}$ on $L$ such that
for any $\lambda_1, \lambda_2 \in \C$, we have 
\begin{equation*}
\label{eq:0.7}
D^{\lambda_1}_L + D^{\lambda_2}_L = D^{\lambda_1 + \lambda_2}_L.
\end{equation*}

Let  $ \cat{M}_{Hod}(L) = \cat{M}_{Hod}(X,n,L)$ denote 
the moduli space parametrising all stable 
$\lambda$-connection $(E, \lambda, D)$ such that$
\bigwedge^{n}E \cong L $ and for every $\lambda \in \C$, 
the $\lambda$-
connection on $\bigwedge^{n}E$ induced by
$D$ coincides with the given $\lambda$-connection 
$D^{\lambda}_{L}$ on $L$.  Let  $\cat{M}'_{Hod}(L) \subset \cat{M}_{Hod}(L)$ be the moduli spaces whose
underlying vector bundle is stable.
We show that $\cat{M}'_{Hod}(L)$ admits a  natural compactification
[see Proposition \eqref{thm:3}] and prove that $Pic(\cat{M}'_{Hod}(L)) \cong \Z$ [see Corollary \eqref{cor:1}. Now, we have a 
 natural surjective algebraic morphism 
 \begin{equation}
 \label{eq:0.8}
 \Phi_L: \cat{M}_{Hod}(L) \to \A^1
 \end{equation}
 defined by $(E,\lambda, D) \mapsto \lambda$.
 Then, we have [see Theorem \eqref{thm:6}]
 \begin{theorem}
\label{thm:0.6}
Any algebraic function on $\cat{M}_{Hod}(L)$ will factor
through the map 
\begin{equation*}
\Phi_L: \cat{M}_{Hod}(L) \to \A^1
\end{equation*}
defined in \eqref{eq:0.8}.
\end{theorem}

Finally, in the last section\eqref{Auto-Hodge}, we
prove that the group $\Aut{\cat{M}'_{Hod}(L)}_0$  of  algebraic 
automorphisms of the moduli space 
$\cat{M}'_{Hod}(L)$ fits into the following short
exact sequence 
\begin{equation*}
\label{eq:0.9}
0 \to \V \xrightarrow{\imath} 
\Aut{\cat{M}'_{Hod}(L)}_0 \xrightarrow{\rho}  \W \times 
\C^*
\to 0
\end{equation*}
of groups [see Theorem \eqref{thm:7}].

\section{Preliminaries}
\label{pre}

\begin{definition}\cite{Si4}
\label{def:1}
A $\lambda$-\emph{connection}  on $X$ is a triple of the 
form $(E, \lambda, D)$, where $\lambda \in \C$, $E$ is 
a holomorphic vector bundle on $X$ of rank $n$ such that
$\det(E)= \Lambda^nE \in J(X)$ and 
\begin{equation*}
\label{eq:1}
D: E \to \Omega^1_X \otimes E
\end{equation*}
is a $\C$-linear map which satisfies (Leibniz property)
\begin{equation}
\label{eq:2}
D(fs)= f D(s) + \lambda  df \otimes s
\end{equation}  
for every local holomorphic section $s$ of $E$  and 
every local section $f$ of $\struct{X}$.
\end{definition} 

If $(E, \lambda, D)$ is a $\lambda$-connection on $X$ 
with $ \lambda \neq 0$, then $D/ \lambda$ is a 
holomorphic connection on $E$. Therefore by Atiyah 
\cite{A} and Weil \cite{W}, degree$(E) = 0$.
On the other hand, if $\lambda = 0$, then $D$ is an 
$\struct{X}$-linear map, and hence $D \in \coh{0}{X}
{\Omega^1_X \otimes \ENd{E}}$. Thus, the triple $(E,0,D)
$ is a Higgs bundle on $X$ with degree$(E) = 0$   
(by Definition \eqref{def:1}). 

Given two $\lambda$-connections $(E,\lambda, D)$ and 
$(E, \lambda, D')$ on $X$, where $E$ and $\lambda$ are fixed, we have $D-D' \in \coh{0}{X}{\Omega^1_X 
\otimes \ENd{E}}$. Conversely, given $\omega \in \coh{0}
{X}{\Omega^1_X \otimes \ENd{E}}$, the triple $(E, 
\lambda, D + \omega)$ is again a $\lambda$-connection
on $X$.

Now, for fixed $E$ and fixed $\lambda \in \C$, let $V_
\lambda(E)$ denote the set consists of 
$\lambda$-connections $(E, \lambda, D)$. Then $V_
\lambda(E)$ is 
an affine space modelled on $\coh{0}{X}{\Omega^1_X 
\otimes \ENd{E}}$ if $\lambda \neq 0$, and $V_0(E) = 
\coh{0}{X}{\Omega^1_X \otimes \ENd{E}}$.

Next, fix $E$. Let $(E, \lambda, D)$ and 
$(E, \lambda', D')$ be two $\lambda$-connections on $X$. 
Then $(E, \lambda + \lambda', D+D')$ is a $\lambda$-
connection on $X$. Moreover, for any $\beta \in \C$, the 
triple $(E, \beta \lambda, \beta D)$ is a $\lambda$-
connection on $X$. The set of all $\lambda$-connections 
in $E$, denoted
by $V(E)$, is a vector space over $\C$. We have 
\begin{equation}
\label{eq:3}
V(E) = \coprod_{\lambda \in \C} V_\lambda(E).
\end{equation}
If $E$ is a stable vector bundle over $X$, then 
using Riemann-Roch theorem and
Serre duality, 
 the dimension 
of $V_0(E) = \coh{0}{X}{\Omega^1_X \otimes \ENd{E}}$
is $n^2(g-1)+1$.

\begin{definition}
\label{def:2}
A $\lambda$-connection $(E,\lambda,D)$ over $X$ is said 
to be semistable (respectively, stable), if for every 
proper subbundle $F \neq 0$
with $rk(F) > 0$,
invariant under  $D$, that is, $D(F) \subset F \otimes 
\Omega^1_X$, we have $degree(F) \leq 0$ (respectively, 
$degree(F)< 0$).

A $\lambda$-connection is polystable if it decomposes as 
a direct sum of stable $\lambda$-connections with the 
same slope.
\end{definition} 
 Let $ \cat{M}_{Hod} = \cat{M}_{Hod}(X,n)$ denote the 
 moduli space of stable $\lambda$-connections of rank $n
 $ on $X$. Then $
 \cat{M}_{Hod}$ is an irreducible  smooth 
 quasi-projective variety of dimension
 $2 n^2(g-1)+2$ (see \cite{Si1}, \cite{Si3}).
  We have
 natural surjective algebraic morphism 
 \begin{equation}
 \label{eq:4}
 \Phi: \cat{M}_{Hod} \to \A^1
 \end{equation}
 defined by $(E,\lambda, D) \mapsto \lambda$. 
 Let $\Phi^\lambda$ denote the inverse image 
 $\Phi^{-1}(\lambda)$. 
  Then $\Phi^1 = \cat{M}_{DR}(X,n)$ and 
 $\Phi^0 = \cat{M}_{Dol}(X,n)$ are de Rham moduli space 
 and Dolbeault moduli space, respectively.
 Let $\cat{M}'_{Hod} \subset \cat{M}_{Hod}$ be the 
 moduli space whose underlying  vector bundle 
 is stable. Then $\cat{M}'_{Hod}$ is a Zariski open
 subvariety of $\cat{M}_{Hod}$ follows from
  [\cite{M}, p.635, Theorem 2.8(B)].

\section{The Picard group of the Hodge moduli space}
\label{Pic}
In this section, we compute the Picard group 
of the Hodge moduli spaces.
Let $\cat{U} = \cat{U}(X,n)$ denote the moduli space of 
 stable vector bundle of rank $n$ and degree $0$, and let 
 \begin{equation}
 \label{eq:5}
 p : \cat{M}'_{Hod} \to \cat{U}
  \end{equation}
be the natural projection sending 
$(E, \lambda, D) \to E$. Then $p$ is a surjective 
algebraic morphism.
For $E \in \cat{U}$, $p^{-1}(E) = V(E)$, where $V(E)$
is the vector space of all $\lambda$-connections in $E$
 that can be expressed as disjoint union of 
 $V_\lambda(E)$ for $\lambda \in \C$ described in 
 $\eqref{eq:3}$. Now, since 
 $E$ is a stable
vector bundle, we have 
\begin{equation}
\label{eq:7}
\dim[\C]{V(E)} = n^2(g-1)+2.
\end{equation}

There is a natural algebraic action  of $\C^*$ on 
$\cat{M}'_{Hod}$. In \cite{Si6}, using this natural 
$\C^*$ action
on $\cat{M}'_{Hod}$, Simpson has proved a relative
copactification of $\cat{M}'_{Hod}$ over $\A^1$.
Nevertheless, we give a different proof for the compactification of $\cat{M}'_{Hod}$, which is constructive in nature. The statement and  the steps involved in the 
following Proposition \eqref{thm:1}
is exactly similar  to the statement and proof of the Theorem 
4.3
in \cite{S}. The main difference is that the fibre of
$p$ defined in \eqref{eq:5} is a finite dimensional 
vector space, while in \cite{S}, the fibre of $p$ 
[see \cite{S} equation(4.1)]
is an affine space modelled over $\coh{0}{X}{\Omega^1_X 
\otimes \ENd{E}}$. Nonetheless, we shall sketch the 
construction of algebraic vector bundle $\cat{Q}$ over
$\cat{U}$, because the morphisms 
\eqref{eq:7.1} and 
\eqref{eq:7.2} in proof of  the following Proposition
\ref{thm:1}  are very crucial in proving Theorem 
\eqref{thm:2}.   

\begin{proposition}
\label{thm:1}
There exists an algebraic vector bundle $\pi: \cat{Q} \to \cat{U}$ such that 
$\cat{M}'_{Hod}$ is embedded in $\p (\cat{Q})$ with $\p (\cat{Q}) \setminus \cat{M}'_{Hod}$ as the
hyperplane at infinity.
\end{proposition}
\begin{proof}

By dualizing the fibres of $p$
defined in \eqref{eq:5},
we construct a new vector bundle  over $\cat{U}$ 
satisfying certain 
conditions. 
 For
any $E \in \cat{U}$, the fiber $p^{-1}(E) = V(E)$ is a 
finite dimensional vector space over $\C$.  
 Let $\pi: \cat{Q} \to \cat{U}$ be an algebraic vector 
 bundle defined as follows;
for every Zariski open subset $U$ of $\cat{U}$, 
\begin{equation}
\label{eq:7.05}
\cat{Q}(U)= \{f: p^{-1}(U) \to {\A^1}~ \mbox{is a regular function} : f|_{V(E)} \in V(E)^{\vee}\}
\end{equation}

 Let $(E, \lambda, D) \in \cat{M}'_{Hod}$.  Then
 define a map $\eta_{(E, \lambda, D)}: V(E)^{\vee} \to \C$, by 
 $\eta_{(E, \lambda, D)}(\varphi) = 
 \varphi[(E,\lambda, D)]$, which is nothing but the 
 evaluation map. Now,
 the kernel $\Ker{\eta_{(E,\lambda, D)}}$ defines a hyperplane in $V(E)^{\vee}$ is
 denoted by $H_{(E, \lambda, D)}$. 
 
 Let $\p (\cat{Q})$ be the projective bundle defined by hyperplanes in the fiber of $\pi$. 
 
   Define a map 
   \begin{equation}
   \label{eq:7.1}
   \iota: \cat{M}'_{Hod} \to \p(\cat{Q})
   \end{equation}
   by sending  $(E, \lambda, D)$ to the hyperplane  $H_{(E, \lambda, D)}$,
  which is clearly an open embedding.

  Set $Y = \p (\cat{Q}) \setminus \cat{M}'_{Hod}$.
  Let
  \begin{equation}
  \label{eq:7.2}
  \tilde{\pi}: \p(\cat{Q}) \to \cat{U}
  \end{equation}
   be the natural projection induced from $\pi$.  
  Then $\tilde{\pi}^{-1}(E) \cap Y$ is a linear hyperplane in the projective space $\tilde{\pi}^{-1}(E)$ for every $E \in \cat{U}$.
\end{proof}

Let $Z$ be a quasi-projective variety over $\C$.  We 
recall an elementary fact about Picard groups.

\begin{lemma}
\label{lem:1}
Let $\iota : U \hookrightarrow Z$ be a Zariski open 
subset of $Z$, whose complement has codimension at least 
$2$ in $Z$.
Then the induced homomorphism  $\iota^*: \Pic{Z} \to  
\Pic{U}$  of Picard groups is an isomorphism.
\end{lemma}
\begin{proof}
See \cite{H}, Chapter II, Proposition 6.5 (b), p.n. 133.
\end{proof}

The morphism $p: \cat{M}'_{Hod} \to \cat{U}$ defined in
\eqref{eq:5} induces a homomorphism 
\begin{equation}
\label{eq:8}
p^*: \Pic {\cat{U}} \to \Pic{\cat{M}'_{Hod}}
\end{equation}
of Picard groups that sends any line bundle $\xi$ over 
$\cat{U}$ to a line bundle $p^* \xi$ over
 $\cat{M}'_{Hod}$.
 Let 
 \begin{equation}
 \label{eq:9}
 \iota: \cat{M}'_{Hod} \hookrightarrow \cat{M}_{Hod} 
\end{equation} 
be the inclusion map.
Then $\iota$ induces a homomorphism
\begin{equation}
\label{eq:10}
\iota^* : \Pic{\cat{M}_{Hod}} \to \Pic{\cat{M}'_{Hod}}
\end{equation}
of Picard groups which restricts any line bundle 
$\xi$ over $\cat{M}_{Hod}$ to $\cat{M}'_{Hod}$.

Now,
we state the following theorem whose first part
(that is, $p^*:\Pic{\cat{U}} \to \Pic{\cat{M}'_{Hod}}$  is an isomorphism) is similar to Theorem 4.4 in \cite{S}, but in different context,
that is for moduli space of logarithmic connections.

\begin{theorem}
\label{thm:2} Let $X$ be a compact Riemann surface of 
genus $g \geq 3$. Then,
the two homomorphism 
\begin{equation}
\label{eq:11}
\Pic{\cat{U}} \xrightarrow{p^*} \Pic{\cat{M}'_{Hod}}
\xleftarrow{\iota^*} \Pic{\cat{M}_{Hod}}
\end{equation}
defined in \eqref{eq:8} and \eqref{eq:10} are isomorphisms.
\end{theorem} 
\begin{proof}
First we show that $\iota^*$ is an isomorphism. In view
of Lemma \eqref{lem:1}, it is sufficient to show that
the compliment 
\begin{equation*}
\label{eq:12}
\cat{M}_{Hod} \setminus \cat{M}'_{Hod} \subset 
\cat{M}_{Hod}
\end{equation*}
is of codimension at least $2$ in $\cat{M}_{Hod}$.
Note that the dimension of the moduli space 
$\cat{M}_{Hod}$ is $2 n^2(g-1)+2$. From the first statement in
\cite{B}[Proposition 1.2, 1], we conclude that the
complement $\cat{M}_{Hod} \setminus \cat{M}'_{Hod} $
has codimension at least $2$ in $\cat{M}_{Hod}$.

Next to show that $p^*$ is an isomorphism. we first 
show that $p^*$ is injective.
Let $\xi \to 
\cat{U}$ be a line bundle such that $p^* \xi$ is a trivial line bundle
over $\cat{M}'_{Hod}$. Giving a trivialization of $p^* \xi$ is equivalent to
giving a nowhere vanishing section of $p^* \xi$ over $\cat{M}'_{Hod}$. Fix 
$s \in \coh{0}{\cat{M}'_{Hod}}{p^* \xi}$ a nowhere vanishing section. Take any
point $E \in \cat{U}$.  Then, from the following commutative diagram 

\begin{equation}
\label{eq:13}
\xymatrix{
p^* \xi \ar[d] \ar[r]^{\tilde{p}} & \xi \ar[d] \\
\cat{M}'_{Hod} \ar[r]^{p} & \cat{U}\\
}
\end{equation}

we get 

\begin{equation*}
\label{eq:14}
s|_{p^{-1}(E)}: p^{-1}(E) \to \xi(E)
\end{equation*}
 a nowhere vanishing map. Notice that $p^{-1}(E) \cong \C^N$ and $\xi(E) \cong \C$, where $N = n^2(g-1)+2$. Now, any nowhere vanishing algebraic function on an affine space
$\C^N$ is a constant function, that is, $s|_{p^{-1}(E)}$ is a constant
function and hence corresponds to a non-zero vector $\alpha_{E} \in \xi(E)$.
Since $s$ is constant on each fiber of $p$,  the trivialization $s$ of $p^*\xi$ descends to a trivialization of the
line bundle $\xi$ over $\cat{U}$, and hence giving a nowhere vanishing
section of $\xi$ over $\cat{U}$.  Thus,  $\xi$ is a trivial line bundle
over $\cat{U}$.

 It remains to show that $p^*$ is surjective. Let 
 $\vartheta \to \cat{M}'_{Hod}$ be an algebraic line 
 bundle. Since  $\cat{M}'_{Hod} \hookrightarrow
 \p (\cat{Q}) $ [follows from  \eqref{eq:7.1}, in the proof  of above
 Proposition \eqref{thm:1}], we can extend $\vartheta$ to a line bundle
 $\vartheta'$ over $\p (\cat{Q})$.
 Again from the morphism $\tilde{\pi}: \p(\cat{Q}) \to \cat{U}$ in \eqref{eq:7.2} in the above Theorem\eqref{thm:1} and from \cite{H}, Chapter III, Exercise 12.5, p.n. 291.,
we have
\begin{equation}
\label{eq:15}
 \Pic {\p (\cat{Q})} \cong \tilde{ \pi}^*\Pic{\cat{U}}\oplus  \Z \struct{\p (\cat{Q})}(1).
\end{equation}
Therefore,
\begin{equation}
\label{eq:16}
\vartheta' = \tilde{\pi}^* L \otimes \struct{\p (\cat{Q})}(l)
\end{equation}
where $L$ is a line bundle over $\cat{U}$ and 
$l \in \Z$.
Since $Y = \p (\cat{Q}) \setminus \cat{M}'_{Hod}$ is the 
hyperplane at infinity, using \eqref{eq:15} the line bundle 
$\struct{\p (\cat{Q})}(Y)$ associated to the divisor $Y$ can be expressed
as 
\begin{equation}
\label{eq:17}
\struct{\p (\cat{Q})}(Y) = \tilde{\pi}^* L_1 \otimes \struct{\p (\cat{Q})}(1)
\end{equation}
for some line bundle $L_1$ over $\cat{U}$.
Now, from \eqref{eq:16} and \eqref{eq:17}, we get
\begin{equation}
\label{eq:18}
\vartheta' = \tilde{\pi}^*(L \otimes 
(L_1^{\vee})^{\otimes l}) \otimes \struct{\p (\cat{Q})}(lY).
\end{equation}
Since, the restriction of the line bundle $\struct{\p 
(\cat{Q})}(Y)$ to the compliment $\p (\cat{Q}) \setminus Y = 
\cat{M}'_{Hod}$ is the trivial line bundle and restriction of $\tilde{\pi}$ to $\cat{M}'_{Hod}$ is the map $p$ defined in \eqref{eq:5}, therefore, we have
\begin{equation}
\label{eq:19}
\vartheta =  p^*(L \otimes (L_1^{\vee})^{\otimes l}).
\end{equation}
This completes the proof.
\end{proof}

Let $\M_{Hod} = \M_{Hod}(X,n)$ (see, \cite{HZ}) denote the coarse moduli 
space of isomorphism classes of polystable 
$\lambda$-connections of rank $n$ on $X$.   Then 
$\M_{Hod}$ is a quasi-projective variety that contains 
$\cat{M}_{Hod}$ as a Zariski open dense subset.

Although, we have assumed that genus $g \geq 3$, but the
following proposition is true for $g \geq 2$.
\begin{proposition}
\label{prop:0}Let $X$ be a compact Riemann surface of
genus $g \geq 2$ and $n \geq 2$.   Then we have 
\begin{equation}
\label{eq:19.2}
\Pic{\M_{Hod}} \cong \Pic{\cat{M}_{Hod}}
\end{equation}
\end{proposition}
\begin{proof}
In view of  Lemma\eqref{lem:1}, it is sufficient to show 
that the codimension of the closed set $W = \M_{Hod} \setminus \cat{M}_{Hod}$ in $\M_{Hod}$
is atleast $2$. And that follows from \cite{HZ},
Proposition 3.1.
\end{proof}

\section{Algebraic functions on the Hodge moduli space}
\label{Alg-fun}

In this section we will study the algebraic functions 
on the Hodge moduli space with fixed determinant.
Fix a line bundle $L$ over the compact Riemann surface 
$X$ of degree zero. For each $\lambda \in \C$, fix a $\lambda$-connection $D^{\lambda}_{L}$ on $L$ such that
for any $\lambda_1, \lambda_2 \in \C$, we have 
\begin{equation}
\label{eq:18.1}
D^{\lambda_1}_L + D^{\lambda_2}_L = D^{\lambda_1 + \lambda_2}_L.
\end{equation}

Let  $ \cat{M}_{Hod}(L) = \cat{M}_{Hod}(X,n,L)$ denote 
the moduli space parametrising all stable $\lambda$-connections $(E, \lambda, D)$ such that

\begin{enumerate}
\item $E$ is a holomorphic vector bundle of rank $n$ 
over $X$ with $\bigwedge^{n}E \cong L $.

\item  for every $\lambda \in \C$, the 
$\lambda$-connection on $\bigwedge^{n}E$ induced by
$D$ coincides with the given $\lambda$-connection 
$D^{\lambda}_{L}$ on $L$.
\end{enumerate}

Then $\cat{M}_{Hod}(L)$ is a smooth quasi-projective
 closed subvariety of $\cat{M}_{Hod}$ of dimension 
 $2(n^2-1)(g-1)+1$. Let $\cat{M}'_{Hod}(L) = \cat{M}
 _{Hod}(L) \cap \cat{M}'_{Hod}$, that is, $
 \cat{M}'_{Hod}(L)$ consists of those points of $\cat{M}
 _{Hod}(L)$ whose underlying vector bundle is stable.
 Then $\cat{M}'_{Hod}(L)$ is a Zariski open dense subset
 of $\cat{M}_{Hod}(L)$ (openness follows from [\cite{M}, p.635, Theorem 2.8(B)]).
 
 Again, we have a
 natural surjective algebraic morphism 
 \begin{equation}
 \label{eq:20}
 \Phi_L: \cat{M}_{Hod}(L) \to \A^1
 \end{equation}
 defined by $(E,\lambda, D) \mapsto \lambda$. 
 Let $\Phi^\lambda_L$ denote the inverse image 
 $\Phi^{-1}_L(\lambda)$. 
 Then, $\Phi^1_L = \cat{M}_{DR}(X,n,L)$ and 
 $\Phi^0_L = \cat{M}_{Dol}(X,n,L)$ are de Rham moduli 
 space and Dolbeault moduli space with fixed determinant 
 $L$, respectively.
 Let $\cat{U}_L$ be the moduli space  
parametrising all the stable holomorphic vector bundle 
$E$   on $X$ of rank $n$ and degree zero such that $\bigwedge^{n}E \cong L$.

Let 
\begin{equation}
\label{eq:22}
p_0: \cat{M}'_{Hod}(L) \to \cat{U}_{L}
\end{equation}
be the projection defined by sending $(E, \lambda, D)$ 
to $E$.  Then  from the
  assumption $\eqref{eq:18.1}$ on the fixed family $(L, \lambda, D^\lambda_L)_{\lambda \in \C}$, for every $E \in \cat{U}_{L}$, 
  $p_0^{-1}(E)$ is a vector space of dimension 
  $(n^2-1)(g-1)+1$.
  Note that if we remove the assumption \eqref{eq:18.1}
  from the fixed family $(L, \lambda, D^\lambda_L)_{\lambda \in \C}$, $p_0^{-1}(E)$ need not
  be a vector space.
   Now, we have following result
  similar to Proposition \eqref{thm:1}.
  
  \begin{proposition}
  \label{thm:3}
  There exists an algebraic vector bundle $\pi': \Xi' \to \cat{U}_L$ such that 
$\cat{M}'_{Hod}(L)$ is embedded in $\p (\Xi')$ with $\p (\Xi') \setminus \cat{M}'_{Hod}(L)$ as the
hyperplane at infinity.
  \end{proposition}
  \begin{proof}
  cf. Proposition \eqref{thm:1}
  \end{proof}
  
  \begin{proposition}
  \label{thm:4}
   The homomorphism $p_0^*: Pic(\cat{U}_{L}) \to 
   Pic(\cat{M}'_{Hod}(L))$ defined by $\xi \mapsto p_0^* 
   \xi$ is an isomorphism of groups.
  \end{proposition}
  \begin{proof}
  The proof is similar to the proof of the Theorem 
  \eqref{thm:2}.
  \end{proof}
  
  Now, from \cite{R}, Proposition 3.4, (ii), we have 
  $Pic(\cat{U}_{L}) \cong \Z$. Thus, in view of 
Theorem \eqref{thm:4}, we have 
\begin{corollary}
\label{cor:1}
$Pic(\cat{M}'_{Hod}(L)) \cong \Z.$
\end{corollary}

Now, consider $\Phi^1_L = \cat{M}_{DR}(X,n,L)$ the 
moduli space of $1$-connection on $X$ with fixed
determinant $(L,1,D^1_L)$ that is nothing but moduli
space of holomorphic connections $(E,1,D)$ on $X$ with fixed 
determinant $(L, D^1_L)$.
Let $\cat{M}'_{DR}(L) \subset \Phi^1_L$ denote the 
Zariski open subvariety such that the underlying
vector bundle is stable (openness follows from 
[\cite{M}, p.635, Theorem 2.8(B)]).

Let 
\begin{equation}
\label{eq:23}
q: \cat{M}'_{DR}(L) \to \cat{U}_L
\end{equation}
be the canonical projection map defined by sending 
$(E,D)$ to $E$. From Atiyah \cite{A} and Weil \cite{W},
being an indecomposable vector bundle of degree zero, any $E \in \cat{U}_L$ admits a holomorphic connection.
Thus $q$ is a surjective morphism.

First note that for any $E \in \cat{U}_{L}$, the 
 holomorphic cotangent space $\Omega^1_{\cat{U}_{L},E}$ at $E$ is isomorphic to $\coh{0}{X}{\Omega^1_X 
 \otimes \ad{E}}$, where $\ad{E} \subset \ENd{E}$ is the 
 subbundle consists of endomorphism of $E$ whose trace is 
 zero. Also, $q^{-1}(E)$ is an affine space modelled 
 over $\coh{0}{X}{\Omega^1_X \otimes \ad{E}}$. Thus, 
 there is a natural action of $\Omega^1_{\cat{U}_{L},E}$  on $q^{-1}(E)$, that is,
\begin{equation*}
\label{eq:24.1}
\Omega^1_{\cat{U}_{L},E} \times q^{-1}(E) \to 
q^{-1}(E)
\end{equation*}
sending $(\omega, D)$ to $\omega + D$, which is faithful
and transitive. Thus, we have (see \cite{S} for the definition of torsor and similar results for instance
\cite{S}
Proposition 4.2)

\begin{proposition}
\label{prop:1}
Let $q:\cat{M}'_{DR}{(L)} \to \cat{U}_{L}$ be the map as defined in 
\eqref{eq:23}. Then $\cat{M}'_{DR}{(L)}$ is a 
$\Omega^1_{\cat{U}_{L}}$-\emph{torsor} on 
${\cat{U}_{L}}$.
\end{proposition}  

\begin{proposition}
\label{prop:2}
 The homomorphism $q^*: Pic(\cat{U}_{L}) \to 
   Pic(\cat{M}'_{DR}(L))$ defined by $\xi \mapsto q^* 
   \xi$ is an isomorphism of groups.
\end{proposition}
\begin{proof}
The arguments involved in the proof is parallel to  the Theorem
\eqref{thm:2}.
\end{proof}

Now, from \cite{R}, Proposition 3.4, (ii), we have $Pic(\cat{U}_{L}) \cong \Z$. Thus, in view of 
Proposition \eqref{prop:2}, we have 
\begin{equation}
\label{eq:24}
Pic(\cat{M}'_{DR}(L)) \cong \Z.
\end{equation}

Let $\Theta$ be the ample generator of the group $Pic(\cat{U}_{L})$. 
We have 
\textbf{Atiyah exact sequence}
\begin{equation}
\label{eq:26}
0 \to \struct{\cat{U}_{L}} \xrightarrow{\imath} 
\At{\Theta} \xrightarrow{\sigma}   T\cat{U}_{L}
\to 0,
\end{equation}
where $\At{\Theta} $ is called
\textbf{Atiyah algebra} of holomorphic line bundle $
\Theta$, and in this case it is equal to $\DifF[1]
{\Theta}{\Theta}$, the sheaf of first order differential
operators on $\Theta$ for more details see \cite{A}.

Dualising the above  Atiyah exact sequence \eqref{eq:26},
we get following exact sequence,

\begin{equation}
\label{eq:27}
0 \to \Omega^1_{\cat{U}_{L}} 
\xrightarrow{\sigma^*} 
\At{\Theta}^* \xrightarrow{\imath^*}   \struct{\cat{U}
_{L}} \to 0
\end{equation}

Consider $\struct{\cat{U}_{L}}$ as trivial line 
bundle $\cat{U}_{L} \times \C$. Let $t: \cat{U}
_{L}\to \cat{U}_{L} \times \C$ be a
holomorphic map defined by $E \mapsto (E,1)$. Then $t$ is
a holomorphic  section of the trivial line bundle $\cat{U}_{L} \times \C$.

Let $T = \Img{t} \subset \cat{U}_{L} \times \C$ be the image of $t$. 
Then $T \to \cat{U}_{L}$ is a fibre bundle.
Consider the inverse image ${\imath^*}^{-1}T \subset 
\At{\Theta}^*$, and denote it by $\cat{C}(\Theta)$.
Using the same technique as in \cite{S}, (see Proposition 5.4), we can show that
\begin{equation}
\label{eq:29}
\cat{C}(\Theta) \cong \cat{M}'_{DR}(L)
\end{equation}

\begin{lemma}
\label{thm:5}
Assume that $ \text{genus}(X) \geq  3$. Then
\begin{equation}
\label{eq:30}
\coh{0}{\cat{M}'_{DR}(L)}{\struct{\cat{M}'_{DR}(L)}} =
\C.
\end{equation}
\end{lemma}
\begin{proof}
In view of 
the above isomorphism \eqref{eq:29} of varieties, it is 
enough to 
show that $\cat{C}(\Theta)$ does not have any 
non-constant algebraic function.  Now, from \cite{BR}, 
Theorem 4.3, p.797 (also see \cite{S}, Theorem 5.7),
we have
\begin{equation*}
\label{eq:31}
\coh{0}{\cat{C}(\Theta)}{\struct{\cat{C}(\Theta)}} =\C.
\end{equation*}
\end{proof}

\begin{proposition}
\label{cor:2}
$\coh{0}{\Phi^1_L}{\struct{\Phi^1_L}} =\C.$
\end{proposition}
\begin{proof}
Since $\cat{M}'_{DR}(L) \subset \Phi^1_L$ is an open dense subvariety. Therefore, from above Lemma \eqref{thm:5}, conclusion follows.
\end{proof}

Next, for any $\lambda \neq 0$, $\Phi^\lambda_L$
is isomorphic to $\Phi^1_L$, by a morphism sending
$(E, \lambda, D)$ to $(E, 1, \lambda^{-1} D)$.
Thus we have
\begin{corollary}
\label{prop:4}
For any $\lambda \neq 0$,
$\coh{0}{\Phi^\lambda_L}{\struct{\Phi^\lambda_L}} =\C.$
\end{corollary}

Now, we are in position to describe the algebraic 
functions on the moduli space  $\cat{M}_{Hod}(L)$,
which generalizes a result (see \cite{BHR} Theorem 4.1, p.1542)  to higher
rank case on the compact Riemann surface $X$.

\begin{theorem}
\label{thm:6}
Any algebraic function on $\cat{M}_{Hod}(L)$ will factor
through the map 
\begin{equation*}
\Phi_L: \cat{M}_{Hod}(L) \to \A^1
\end{equation*}
defined in \eqref{eq:20}.
\end{theorem}
\begin{proof}
Let $g: \cat{M}_{Hod}(L) \to \C$ be an algebraic function. Then, for each $\lambda \neq 0$, the restriction of $g$ to each fibre $\Phi^\lambda_L$ of 
$\Phi_L$ is a constant function, follows from Corollary
\eqref{prop:4}. Since 
\begin{equation}
\label{eq:32}
\coprod_{0 \neq \lambda \in \C}
\Phi^\lambda_L \hookrightarrow \cat{M}_{Hod}(L)
\end{equation}
is a dense open subset of $\cat{M}_{Hod}(L)$,
  and hence defining a function from $\A^1 \to
\A^1$.
\end{proof}

Let $\M_{Hod}(L) = \M_{Hod}(X,n,L)$ denote the
moduli 
space of isomorphism classes of polystable 
$\lambda$-connections of rank $n$ on $X$ such that 
\begin{enumerate}
\item $E$ is a holomorphic vector bundle of rank $n$ 
over $X$ with $\bigwedge^{n}E \cong L $.

\item  for every $\lambda \in \C$, the 
$\lambda$-connection on $\bigwedge^{n}E$ induced by
$D$ coincides with the given $\lambda$-connection 
$D^{\lambda}_{L}$ on $L$.
\end{enumerate}

Then 
$\M_{Hod}(L)$ is a quasi-projective variety that contains 
$\cat{M}_{Hod}(L)$ as a Zariski open dense subset, and 
we have the canonical map 
\begin{equation}
\label{eq:32.1}
\tau: \M_{Hod}(L) \to \A^1
\end{equation}
defined by sending $(E, \lambda, D) \mapsto \lambda $.
\begin{corollary}
\label{cor:3}Any algebraic function on $\M_{Hod}(L)$ will factor
through the map 
\begin{equation*}
\tau: \M_{Hod}(L) \to \A^1
\end{equation*}
defined in \eqref{eq:32.1}.
\end{corollary}

\section{Automorphism of Hodge moduli space}
\label{Auto-Hodge}

In \cite{BH}, the automorphism group of moduli space
of  rank
one $\lambda$-connections have been studied. In this 
section we will obtain a short exact sequence which
contains the group of algebraic automorphism of Hodge 
moduli space.  
  The group
of algebraic automorphisms of the moduli spaces
$\cat{M}'_{Hod}(L)$  and $\cat{U}_L$ is denoted by $
\Aut{\cat{M}'_{Hod}(L)}$ and $\Aut{\cat{U}_L}$, 
respectively. Let 
\begin{equation*}
\label{eq:33}
\Aut{\cat{M}'_{Hod}(L)}_0 \subset \Aut{\cat{M}'_{Hod}(L)} ~~ \mbox{and} ~~ \Aut{\cat{U}_L}_0 \subset \Aut{\cat{U}_L}
\end{equation*}
be the connected component of the respective groups of algebraic automorphisms containing the identity 
automorphism.

Let $\Psi: \cat{M}'_{Hod}(L) \to \A^1$ be the canonical
projection defined by $(E, \lambda, D) \mapsto \lambda$.
Then, for each $\lambda \in \C$, we denote the inverse
image $\Psi^{-1}(\lambda)$ by $\Psi^\lambda$.
For each $\lambda \in \C$, we have canonical projection
\begin{equation}
\label{eq:34}
f^\lambda: \Psi^\lambda \to \cat{U},
\end{equation}
defined by $(E, \lambda, D) \to E$.
Let $\V$ consists of those 
$T \in \Aut{\cat{M}'_{Hod}(L)}_0$
such that for every $\lambda \in \C$, we have 
\begin{enumerate}
\item $T|_{\Psi^\lambda}: \Psi^\lambda \to \Psi^\lambda$.
\item $f^\lambda \circ T|_{\Psi^\lambda} = f^\lambda$.
\end{enumerate}
Then $\V$ is a subgroup of $\Aut{\cat{M}'_{Hod}(L)}_0$.
Let $\W$ denote the group of all algebraic maps from
$\C$ to $\Aut{ \cat{U}_L}_0$.  Then we have following

\begin{theorem}
\label{thm:7}
The group $\Aut{\cat{M}'_{Hod}(L)}_0$ fits in a short
exact sequence 
\begin{equation}
\label{eq:35}
0 \to \V \xrightarrow{\imath} 
\Aut{\cat{M}'_{Hod}(L)}_0 \xrightarrow{\rho}  \W \times 
\C^*
\to 0
\end{equation}
of groups.
\end{theorem}
\begin{proof}
Let $T \in \Aut{\cat{M}'_{Hod}(L)}_0$. From Proposition
\eqref{prop:4}, $\Psi^{\lambda}$ does not
admit any nonconstant algebraic function, for any $\lambda \in \C^*$. Hence the composition
\begin{equation}
\label{eq:36}
\Psi^\lambda \xrightarrow{T|\Psi^\lambda} 
\cat{M}'_{Hod}(L) \xrightarrow{\Psi}  \A^1,
\end{equation}
is a constant function for all $\lambda \in \C^*$.
Note that for every $\lambda  \in \C^*$, $\Psi^0$ is not
isomorphic to $\Psi^\lambda$, therefore
\begin{equation}
\label{eq:37}
\Psi (T (\Psi^\lambda)) \in \C^*.
\end{equation}
Now, $T$ being an automorphism, \eqref{eq:37} holds for
$T^{-1}$, and hence the composition in \eqref{eq:36}
is a constant function for every $\lambda \in \C$.
This implies that there is an algebraic automorphism
$\gamma_0: \A^1 \to \A^1$
such that the following diagram 

\begin{equation}
\label{eq:38}
\xymatrix{
\cat{M}'_{Hod}(L)  \ar[d]^\Psi \ar[r]^{T} & 
\cat{M}'_{Hod}(L) \ar[d]^\Psi \\
\A^1 \ar[r]^{\gamma_0} & \A^1\\
}
\end{equation}
commutes.
Since $\Psi^0$ is not isomorphic to $\Psi^\lambda$
for every $\lambda \in \C^*$, therefore 
$\gamma_0(0) = 0$.
Thus, $\gamma_0(z) = \gamma.z$  for some 
$\gamma \in \C^*$ and $z \in \A^1$. Clearly, $\gamma \in\C^*$ depends on $T$.

For any $\lambda \in \C$, let 
\begin{equation}
\label{eq:39}
\psi^\lambda: \Psi^\lambda \to \cat{U}_L
\end{equation}
be the canonical projection defined by sending
$(E,\lambda, D) \to E$.

For $\lambda \neq 0$ the above map is surjecive, and 
the fibres of $\psi^\lambda$ is an affine space modelled
over  $\coh{0}{X}{\Omega^1_X \otimes \ENd{E}}$ and
for $\lambda = 0$, the inverse image $(\psi^0)^{-1}(E)
= \coh{0}{X}{\Omega^1_X \otimes \ENd{E}}$.

Thus, for given $T \in \Aut{\cat{M}'_{Hod}(L)}_0$, there 
exists a unique automorphism
\begin{equation}
\label{eq:40}
\Upsilon(\lambda): \cat{U}_L \to \cat{U}_L
\end{equation}
such that the following diagram 
\begin{equation}
\label{eq:41}
\xymatrix{
\Psi^\lambda  \ar[d]^{\psi^\lambda} \ar[r]^{T|_{\Psi^
\lambda}} & 
\Psi^{\gamma \lambda } \ar[d]^{\psi^{\gamma \lambda}} \\
\cat{U}_L \ar[r]^{\Upsilon(\lambda)} & \cat{U}_L\\
}
\end{equation}
commutes, for every $\lambda \in \C$.
Also, note that $\Upsilon(\gamma) \in \Aut{\cat{U}_L}_0$.
From above phenomenon, we define
\begin{equation}
\label{eq:42}
\rho: \Aut{\cat{M}'_{Hod}(L)}_0 \to \W \times \C^*
\end{equation}
by  $T \mapsto (\Upsilon, \gamma)$.
It is straightforward to check that $\rho$ is a 
surjective group homomorphism.
Now, the kernel of $\rho$ is $\V$, which gives the 
short exact sequence \eqref{eq:35} of groups.

\end{proof}

\end{document}